\documentclass[10pt]{amsart}
\usepackage{amsmath,amscd,amsthm,amssymb,enumerate, url, eucal}
\usepackage{fullpage}

\newtheoremstyle{nodot}
{3pt}
{3pt}
{\itshape}
{}
{\bfseries} 
{}
{0pt}
{\thmnote{#3}}

\swapnumbers \theoremstyle{plain}
\newtheorem{theorem}{Theorem}[section]

\newtheorem{corollary}[theorem]{Corollary}

\theoremstyle{definition}

\newtheorem{definitions}[theorem]{Definitions}
\newtheorem{history}[theorem]{Historical remarks}

\newtheorem{notation}[theorem]{Notation}

\newtheorem{remark}[theorem]{Remark}
\newtheorem{key}[theorem]{Key points}

\def \reals {\mathbb{R}}

\def \naturals{\mathbb{N}}
\def \integers {\mathbb{Z}}

\def\into{\hookrightarrow}
\def\onto{\twoheadrightarrow}

\def\d1{\discretionary{-}{}{-}}

\renewcommand{\le}{\leqslant}
\renewcommand{\ge}{\geqslant}

\begin{document}

\title{A mnemonic for the  graded-case Golod-Shafarevich inequality}

\author{David  Anick}
\address{Laboratory for Water and Surface Studies, Department of Chemistry, Tufts University,
62 Pearson Rd., Medford MA02155, USA}
\email{david.anick@rcn.com}

\author{Warren Dicks*}
\address{Departament de Matem{\`a}tiques, Universitat Aut{\`o}noma de Barcelona, E-08193 Bellaterra (Barcelona), SPAIN}
\email{dicks@mat.uab.cat}
\thanks{*Corresponding author;  research supported by
MINECO (Spain) through project number MTM2014-53644}


\begin{abstract}
We draw attention to an easy-to-remember explanation for the graded-case inequality of Golod and Shafarevich.
We review, unify, and simplify some of the classic material on this inequality, thereby offering a new,
concise exposition for it.
\end{abstract}

\maketitle

Let $K$ be a field, and  \mbox{$B= \bigoplus\limits_{n \in \integers} B_n = K \langle X \mid R \rangle$} be
 a  $\integers$-graded,  associative  $K$-algebra that is presented with
a generating set $X$ and a relating set $R$, both of which are positively graded.
For each \mbox{$n\in \integers$}, set \mbox{$\mathbf{b}_n:= \dim_K(B_n)$}.
One form of the \textit{graded-case inequality of Golod and Shafarevich} is
\begin{equation}\label{eq:ineq}
\textstyle (\forall n \in \integers) \qquad \sum\limits_{x \in X} \mathbf{b}_{n-\deg(x)} \le
 (\sum\limits_{r \in R} \mathbf{b}_{n-\deg(r)})  +  \mathbf{b}_n. \vspace{-1mm}
\end{equation}
 Many important applications of this inequality can be found on its Wikipedia page~\cite{Wiki}.

The following is an easy\d1to\d1remember  explanation for this inequality.
The Koszul resolution\vspace{-1mm}
\begin{equation}\label{eq:ex}
\textstyle 0 \to \operatorname{Ker} \partial \to \bigoplus\limits_{r \in R} B
 \xrightarrow{\partial} \bigoplus\limits_{x \in X}  B  \to B \to K\to 0
\end{equation}
respects $\integers$-gradings;
hence, for each \mbox{$n \in \integers$}, the $n$th component of~\eqref{eq:ex}
is an exact sequence of $K$-modules
\begin{equation}\label{eq:graded}
\textstyle 0_n \to ( \operatorname{Ker}  \partial)_n  \to  \bigoplus\limits_{r \in R} B_{n-\deg(r)}
  \xrightarrow{\partial_n}
\bigoplus\limits_{x \in X}  B_{n-\deg(x)}  \to B_n \to K_n \to 0_n.\vspace{-2pt}
\end{equation}
The middle part is an exact sequence of $K$-modules\vspace{-1mm}
\begin{equation}\label{eq:ex2}
\textstyle \bigoplus\limits_{r \in R} B_{n-\deg(r)}
  \xrightarrow{\partial_n} \bigoplus\limits_{x \in X}  B_{n-\deg(x)}  \to B_n,
\end{equation}
and,  because $K$ is a field, the $K$-dimension of the inner term is at most the sum of the $K$-dimensions
of the two outer terms.  Hence,~\eqref{eq:ineq} holds, and there are not even any cardinality restrictions.

This concludes the main point of this note, but perhaps
lengthy explanatory remarks are in order.  As far as we know,   the preceding argument, which we
chanced upon in 1982, has
not appeared in print before now, and we
would be   interested to hear from anyone who knows that it has.
 There exist proofs of~\eqref{eq:ineq} in the literature which go via~\eqref{eq:graded} without mentioning~\eqref{eq:ex}
or~\eqref{eq:ex2};
see, for example, the original source~\cite{GS}, or Theorem~8.1.1.1 of~\cite{H}.
 There is a proof of  a special case of~\eqref{eq:ineq}, Theorem 2.3.4(i) of~\cite{EE}, that goes
via~\eqref{eq:ex}, bypassing ~\eqref{eq:graded} and~\eqref{eq:ex2}.
There also exist proofs in the literature which have~\eqref{eq:ex},~\eqref{eq:graded},
and~\eqref{eq:ex2} in  the background; see, for example,   Section~3.5 of~\cite{U}.

In what follows, with an eye toward how the material might be presented or taught as a coherent unit,
we provide a digest of five  topics: the   Golod-Shafarevich  $p$-group theorem;
the construction of the Koszul resolution for augmented algebras~\eqref{eq:ex};
its graded version~\eqref{eq:graded};  the Hilbert series form  of~\eqref{eq:ineq}; and,
the group-algebra analogue of the Koszul resolution.

We are grateful to Andrei Jaikin, Clas L\"ofwall, Dmitry Piontkovskii,  Jan-Erik Roos, and John Wilson for their expert
advice concerning the literature.

\section{The Golod-Shafarevitch $p$-group theorem}

As Bourbaki intended, we let \mbox{$\naturals$} denote the set of finite cardinals,
 \mbox{$\{0,1,2,3,\ldots,\}$}.

The following evolved through work of  Golod, Shafarevich, Gasch\"utz, Vinberg, and Serre;
 it may not have been expressed in this form before.

\begin{theorem} \label{thm:exact}  Let $K$ be a field, \mbox{$B = K \oplus \mathfrak{b}$} be an
augmented $K$-algebra, $X$ be a generating set for \mbox{$\mathfrak{b}$} as left $B$-module,
and $R$ be a relating set for \mbox{$\mathfrak{b}$} when generated by $X$.
   If   \, \mbox{$1 \le \vert X \vert =   \dim_K( \mathfrak{b}/ \mathfrak{b}^2) < \aleph_0 $}     and
 \mbox{$\vert R \vert \le   \frac{1}{4} \vert X \vert^2$}, then \mbox{$\dim_K(B) \ge \aleph_0$}.
\end{theorem}

\begin{proof}[Proof \normalfont (after Serre~\cite{S})] We have an exact left-$B$-module sequence
\begin{equation}\label{eq:needed}
\textstyle \bigoplus\limits_{R} B  \xrightarrow{\partial}
\bigoplus\limits_{X} B  \xrightarrow{\pi}  \mathfrak{b} \to 0, \vspace{-1mm}
\end{equation}
 where   \mbox{$\bigoplus\limits_{R} B$} denotes the direct sum of  copies of~$B$ indexed by $R$.

Let $n$ range over \mbox{$\naturals$}.

We first use the hypotheses that \mbox{$\vert X \vert = \dim_K( \mathfrak{b}/ \mathfrak{b}^2) < \aleph_0$}.
 Since $K$ is a field,  the surjective map\vspace{-1mm}
$$\textstyle   \bigoplus\limits_{X} (B/\mathfrak{b})   \xrightarrow{( B /{\mathfrak{b}})\otimes_B \pi}
\mathfrak{b}/ \mathfrak{b}^2$$ is injective.
Hence,
 \mbox{$\operatorname{Ker} \pi \subseteq    \bigoplus\limits_{X} \mathfrak{b}$}.  By the exactness of~\eqref{eq:needed},
\mbox{$ \partial (  \bigoplus\limits_{R} B ) = \operatorname{Ker} \pi \subseteq
 \bigoplus\limits_{X} \mathfrak{b}$}.
 By the left $B$-linearity of~$\partial$,
\mbox{$ \partial(  \bigoplus\limits_{R}  \mathfrak{b}^{n-1} )
\subseteq  \bigoplus\limits_{X}  \mathfrak{b}^{n}$}, where we define \mbox{$\mathfrak{b}^{-1}: = \mathfrak{b}^{0}: = B$}
and \mbox{$\mathfrak{b}^{n+1}:= \mathfrak{b}^{n}{\cdot} \mathfrak{b}$}.
On applying \mbox{$(B/\mathfrak{b}^{n} )\otimes_B -$} to~\eqref{eq:needed},
 we obtain  an exact left-$B$-module sequence\vspace{-2mm}
 $$\textstyle\bigoplus\limits_{R}  (B/\mathfrak{b}^{n})  \xrightarrow{\overline{\partial}}
\bigoplus\limits_{X} (B/\mathfrak{b}^{n})  \to \mathfrak{b}/\mathfrak{b}^{n+1} \to 0  \vspace{-1mm}$$
such that
\mbox{$\textstyle\overline{\partial}\bigl(\bigoplus\limits_{R}  (\mathfrak{b}^{n-1}/\mathfrak{b}^{n})\bigr)
 =  \{0\}.$} There is then induced an exact   left-$B$-module sequence \vspace{-2mm}
$$\textstyle\bigoplus\limits_{R} (B/\mathfrak{b}^{n-1})  \to
\bigoplus\limits_{X} (B/\mathfrak{b}^{n})  \to \mathfrak{b}/\mathfrak{b}^{n+1} \to 0. \vspace{-1mm}$$
Set \mbox{$\mathbf{a}_{n-2}:= \dim_K(B/\mathfrak{b}^{n-1})$}. Since $K$ is a field,  \mbox{$\textstyle
\mathbf{a}_{n}  {-} 1  \le \vert X \vert  {\cdot} \mathbf{a}_{n-1}  \le
\vert R \vert   {\cdot} \mathbf{a}_{n-2}  +  \mathbf{a}_{n}  {-}1.$}
By induction on $n$,
\mbox{$\mathbf{a}_n \le \sum\limits_{i=0}^{n} \vert X \vert^i < \aleph_0$}. \vspace{-2mm}

We next use the hypothesis that  \mbox{$\vert X \vert^2 - 4 \vert R \vert \ge 0$}.
Set \mbox{$\lambda:= \frac{\vert X \vert - \sqrt{\vert X \vert^2 - 4 \vert R \vert}}{2}$}
and \mbox{$\mu:= \frac{\vert X \vert  + \sqrt{\vert X \vert^2 - 4 \vert R \vert}}{2}$}.   Then
 \mbox{$0 \le \lambda \le \mu$}.
Set  \mbox{$\mathbf{b}_{n-1}:= \mathbf{a}_{n-1} {-} \lambda{\cdot} \mathbf{a}_{n-2} \in \reals$}.
Then $$\mathbf{b}_{-1} = \mathbf{a}_{-1} {-} \lambda{\cdot} \mathbf{a}_{-2} = 0\,\, \,\text{ and }\,\,\,\,
 \mathbf{b}_{n} {-} \mu{\cdot}\mathbf{b}_{n-1} =
 \mathbf{a}_{n} {-} (\lambda+\mu){\cdot}  \mathbf{a}_{n-1} + \mu{\cdot} \lambda{\cdot} \mathbf{a}_{n-2}
= \mathbf{a}_{n} {-}  \vert X \vert{\cdot}  \mathbf{a}_{n-1} + \vert R \vert{\cdot} \mathbf{a}_{n-2}  \ge 1,\vspace{-1mm}
$$
by the previous paragraph.  By induction on $n$,  \mbox{$ \mathbf{b}_n \ge \sum\limits_{i=0}^{n} \mu^i$}.  Now,
 \mbox{$\sum\limits_{i=0}^{n} \mu^i \le  \mathbf{b}_n = \mathbf{a}_{n} {-}
 \lambda{\cdot} \mathbf{a}_{n-1} \le \mathbf{a}_{n}$}.\vspace{-2mm}
See also Remark~\ref{rem:serre} below.

It remains to use the hypothesis \vspace{-1mm} that  \mbox{$\vert X \vert  \ge
1$}.  Here,  \mbox{$\mu \ge 1$}, for if
 \mbox{$\vert X \vert   < 2$}, then  \mbox{$\vert X \vert   =1$}, hence
\mbox{$0 =  \lfloor\frac{1}{4} \vert X \vert ^2\rfloor \ge \vert R \vert$}, and, hence, \mbox{$\mu =1$}.
Now \mbox{$n{+}1 \,\le \,  \sum\limits_{i=0}^{n} \mu^i \, \le \,\mathbf{a}_n \le \dim_K(B)$}  and, hence,  \mbox{$\dim_K(B) \ge
\aleph_0$}.
\end{proof}

\vspace{-4mm}

\begin{remark}  In the foregoing proof,
 \mbox{ $\textstyle\sum\limits_{i=0}^{n} \vert X \vert^i  \,\, \ge\,\,
\mathbf{a}_n  \,\,\ge \,\, \sum\limits_{i=0}^{n} \mu^i.$}
If \mbox{$(\vert X\vert, \vert R \vert)$} is neither \mbox{$(1,0)$} nor  \mbox{$(2,1)$},
then \mbox{$\mu > 1$} and, hence, the growth rate of \mbox{$\mathbf{a}_n$} is exponential.
\end{remark}

\begin{history} For  details about the following, see~\cite{GS} and~\cite{S}.

Let $p$ be a prime number, and $G$ be a nontrivial,
finite $p$-group.  Set \mbox{$K:= \integers/p \integers$} and   \mbox{$B:=KG$}, the group algebra.
Let \mbox{$\mathfrak{b}$} denote the kernel of  the $K$-algebra homomorphism
\mbox{$B \to K$} which carries $G$ to~$\{1\}$.
For each \mbox{$n \in \naturals$}, set \mbox{$\mathbf{d}_n:= \dim_{K} \bigl(\operatorname H_n(G,K)\bigr)$}.
Recall that \mbox{$\operatorname H_1(G,K) =  \mathfrak{b}/ \mathfrak{b}^2$}.
 From the theory of minimal resolutions, it is known that
there exist exact  left-$B$-module sequences of the form
\mbox{$ \cdots \to   B^{\mathbf{d}_3}  \to  B^{\mathbf{d}_2}   \to
 B^{\mathbf{d}_1} \to \mathfrak{b} \to 0.$}
By Theorem~\ref{thm:exact},  \mbox{$\mathbf{d}_2 > \frac{1}{4}\mathbf{d}_1^2$}.

It is known that   \mbox{$\mathbf{d}_1$} equals the minimum
number of elements it   takes to generate $G$ as  a pro-$p$ group,
and that for any generating set  of  \mbox{$\mathbf{d}_1$} elements,
\mbox{$\mathbf{d}_2$} equals the minimum
number of relations it  takes to present $G$ as  a pro-$p$ group.
(By the Burnside basis theorem, \mbox{$\mathbf{d}_1$} equals the minimum
number of elements it   takes to generate $G$ as  a  group.
For any generating set  of  \mbox{$\mathbf{d}_1$} elements,
the minimum number of relations it  takes to present $G$ as  a group is at least
 \mbox{$\mathbf{d}_2$}, but it is not known if equality holds.)

 The main objective of Golod and Shafarevich in~\cite{GS}, and the reason for which~\eqref{eq:ineq} was first developed,
was to prove that  \mbox{$\mathbf{d}_2 > \frac{1}{4}(\mathbf{d}_1{-}1)^2$}.
 It followed from this,  together with an earlier result of Shafarevich,
 that the   class-field-tower problem had a negative solution,
that is, there do exist  infinite class-field towers.
Gasch\"utz and Vinberg~\cite{V}  independently   refined
 the inequality to   \mbox{$\mathbf{d}_2 > \frac{1}{4}\mathbf{d}_1^2$}.
 Serre~\cite{S} gave the above proof of this refined inequality.
Nevertheless, there still remain many applications of~\eqref{eq:ineq} which have
not been superseded.
\end{history}

\section{The Koszul resolution for an augmented algebra}\label{sec:one}

\begin{notation}\label{not:augmented} Let $K$ be a field,  $X$ be a set,   $F$ be the free associative $K$-algebra on $X$,
and $\mathfrak{f}$ be the two-sided
ideal of $F$ generated by $X$.
 We write \mbox{$F = K \langle X \rangle$}.

Let $R$ be a family of elements of $\mathfrak{f}$, possibly with repetitions,
 and   $\mathfrak{r}$ denote the two-sided ideal of $F$
generated by the elements of $R$.  Set \mbox{$B := F/\mathfrak{r}$}  and
\mbox{$\mathfrak{b}:= \mathfrak{f}/\mathfrak{r}$}.
In summary, \mbox{$B= K \oplus \mathfrak{b}$} is an augmented associative $K$-algebra  presented with generating set $X$
and relating set~$R$.  We write   \mbox{$B = K \langle X \mid R \rangle$}.

Set
\mbox{$K^{(X)} := \bigoplus\limits_{x \in X}Kx$},\vspace{1mm}
\mbox{$F^{(X)}:= F \otimes_K K^{(X)}$},  and \mbox{$B^{(X)}:= B \otimes_K K^{(X)}$}; these are
the free left modules  on~$X$ over $K$, $F$, and $B$, respectively.
Similar notation will apply\vspace{1mm} with $R$ in place of $X$.  At one stage,
we shall use the natural $K$-centralizing $K$-bimodule structure of \mbox{$K^{(R)}$}.
\end{notation}

\begin{definitions}\label{defs:augmented} Each element $f$ of $\mathfrak{f}$ has a unique
expression as a left
$F$-linear combination of the elements of $X$, and we shall write this as
\mbox{$f= \sum\limits_{x \in X} \frac{\partial f}{\partial x} {\cdot} x$}.

We have an isomorphism of left $F$-modules
\begin{equation*}
\textstyle  \mathfrak{f}  \xrightarrow{\sim}  F^{(X)}, \qquad
 f = \sum\limits_{x\in X} \frac{\partial f}{\partial x}
{\cdot}  x \mapsto \textstyle \sum\limits_{x\in X} \frac{\partial f}{\partial x} \otimes x.\vspace{-1mm}
\end{equation*}
On applying    \mbox{$(F/\mathfrak{r})\otimes_F -$}, we obtain an isomorphism of left
 \mbox{$ F/\mathfrak{r}$}-modules  $$\textstyle \mathfrak{f}/\mathfrak{rf}  \xrightarrow{\sim}   B^{(X)},
\qquad f + \mathfrak{rf} \mapsto \textstyle \sum\limits_{x\in X}
 (\mkern-2mu\frac{\partial f}{\partial x}+\mathfrak{r} )\otimes x.$$
\vspace{-4mm}

We  have also a surjection of $F$-bimodules
\begin{equation*}
\textstyle F \otimes_K K^{(R)} \otimes_ K F  \onto \mathfrak{r}, \qquad
f_1 \otimes r \otimes f_2 \mapsto f_1{\cdot} r{\cdot} f_2.
\end{equation*}
On applying    \mbox{$(F/\mathfrak{r}) \otimes_F - \otimes_F (F/\mathfrak{f})$},  we obtain
a surjection of left \mbox{$F/\mathfrak{r}$}-modules
 $$\textstyle B^{(R)}
  \onto    \mathfrak{r}/\mathfrak{rf}, \qquad   (f+\mathfrak{r})\otimes r \mapsto f{\cdot} r + \mathfrak{rf}.$$

The cokernel of the composite \mbox{$ B^{(R)}
  \onto    \mathfrak{r}/\mathfrak{rf} \into  \mathfrak{f}/\mathfrak{rf} \xrightarrow{\sim}   B^{(X)}$}
is isomorphic to \mbox{$\mathfrak{f}/\mathfrak{r}$}, which is \mbox{$\mathfrak{b}$}.  We then have
an exact left-$B$-module sequence\vspace{-1mm}
\begin{equation}\label{eq:exact}
\textstyle  B^{(R)}  \xrightarrow{b\otimes r \mapsto \sum\limits_{x \in X}
b{\cdot} (\mkern-2mu\frac{\partial r}{\partial x}+ \mathfrak{r})\otimes   x}
 B^{(X)}
\xrightarrow{ b \otimes x \mapsto b{\cdot} (x+\mathfrak{r})} \mathfrak{b} \to 0.
\end{equation}
On splicing~\eqref{eq:exact} and
\mbox{$0 \to \mathfrak{b} \to B  \to B/\mathfrak{b} \to 0$},
we obtain what we  call \textit{the Koszul resolution}\vspace{-1mm}
\begin{equation*}
\textstyle 0 \to \operatorname{Ker} \partial \to  B^{(R)}  \xrightarrow{\partial:b\otimes r
\mapsto \sum_{x \in X}\mkern-4mu b{\cdot} (\mkern-2mu\frac{\partial r}{\partial x}+ \mathfrak{r})\otimes   x}
 B^{(X)}
\xrightarrow{ b \otimes x \mapsto b{\cdot} (x+\mathfrak{r})} B  \to B/\mathfrak{b} \to 0.
\end{equation*}
The part  that interests us is\vspace{-4mm}
\begin{equation}\label{eq:exact3}
\textstyle  B^{(R)}  \xrightarrow{b\otimes r \mapsto \sum\limits_{x \in X} b{\cdot}
(\mkern-2mu\frac{\partial r}{\partial x}
+ \mathfrak{r})\otimes   x}
 B^{(X)}
\xrightarrow{ b \otimes x \mapsto b{\cdot} (x+\mathfrak{r})} B.
\end{equation}
\end{definitions}

\begin{remark}
Suppose that the induced map \mbox{$(B/\mathfrak{b})^{(X)} \to \mathfrak{b}/ \mathfrak{b}^2$} is bijective
or, equivalently,  that each element of  $R$ lies in \mbox{$\mathfrak{f}^2$}.  If $X$ is a finite, nonempty set
and \mbox{$\vert R \vert  \le   \frac{1}{4} \vert X \vert ^2 $},  then applying
Theorem~\ref{thm:exact}   to~\eqref{eq:exact} shows that    \mbox{$\dim_K(B) = \aleph_0$}.
\end{remark}

\section{The graded case of the Koszul resolution}\label{sec:graded}

Continuing with the notation developed in Section~\ref{sec:one}, we now hypothesize
 a $\integers$-graded $K$-algebra structure for~$B$, as follows.

Let \mbox{$\deg:X \to \naturals {-}\{0\}$},
\mbox{$x \mapsto \deg(x)$}, be any map; there is then an induced
$\integers$-graded $K$-algebra structure  \mbox{$F= \bigoplus\limits_{n \in \integers} F_n$}
with \hskip-1mm \mbox{$\bigoplus\limits_{n \in \integers-\naturals} \hskip-2mm F_n = \{0\}$},
\mbox{$F_0 = K$},
\hskip-3.3mm \mbox{$\bigoplus\limits_{n \in \naturals {-}\{0\}} \hskip-3.3mm F_n = \mathfrak{f}$}, and
\mbox{$x \in F_{\deg(x)}$}  for  each   \mbox{$x \in X$}.

We henceforth restrict to the case where each element of $R$ lies in \hskip-3.3mm
\mbox{$\bigcup\limits_{n \in \naturals {-}\{0\}}\hskip-3.3mm  F_n$}.
There is then an induced $\integers$-graded $K$-algebra structure
\mbox{$B= \bigoplus\limits_{n \in \integers} B_n$} with \hskip-1mm
\mbox{$\bigoplus\limits_{n \in \integers-\naturals} \hskip-2mm B_n = \{0\}$}, \mbox{$B_0 = K$},
\hskip-3.3mm  \mbox{$\bigoplus\limits_{n \in \naturals {-}\{0\}}\hskip-3.3mm B_n = \mathfrak{b}$},
 and \mbox{$x + \mathfrak{r} \in B_{\deg(x)}$}
for  each   \mbox{$x \in X$}. We choose a map  \mbox{$\deg:  R\to \naturals{ -}\{0\}$},
\mbox{$r \mapsto \deg(r)$}, such that \mbox{$r \in F_{\deg(r)}$};
thus, as in~\cite{GS},  each occurrence  of $0$ in $R$   has some positive finite degree.
Notice that \mbox{$\frac{\partial r}{\partial x} \in F_{\deg(r)-\deg(x)}$}.

Let $n$ range over   \mbox{$\integers$}.   Set \mbox{$\mathbf{b}_n:= \dim_K(B_n)$}.
Now \eqref{eq:exact3} gives an exact sequence of degree-$n$ $K$-modules
\begin{equation*}
\textstyle \bigoplus\limits_{r \in R} (B_{n-\deg(r)} \otimes_K Kr) \xrightarrow{b \otimes r \mapsto \sum\limits_{x \in X}
b {\cdot}(\mkern-2mu\frac{\partial r}{\partial x} + \mathfrak{r}) \otimes x}
 \bigoplus\limits_{x \in X} (B_{n-\deg(x)} \otimes_K K x) \xrightarrow{ b \otimes x \mapsto b{\cdot}(x+\mathfrak{r})}
B_n.
\end{equation*}
  Since $K$ is a field, we have one form of the Golod-Shafarevich inequality:
\begin{equation}\label{eq:ineqbis}
\textstyle \sum\limits_{x \in X} \mathbf{b}_{n-\deg(x)} \le
 (\sum\limits_{r \in R} \mathbf{b}_{n-\deg(r)})  +  \mathbf{b}_n. \vspace{-1mm}
\end{equation}
Set  \mbox{$X_n:=  \{x \in X : \deg(x) = n\}$} and
\mbox{$\mathbf{x}_n := \left\vert X_n \right\vert$}.  Then
$$\textstyle\sum\limits_{x \in X} \mathbf{b}_{n-\deg(x)}= \sum\limits_{i \in \integers}
 \sum\limits_{x \in X_i} \mathbf{b}_{n-\deg(x)}
=   \sum\limits_{i \in \integers}   \mathbf{x}_i{\cdot} \mathbf{b}_{n-i}.\vspace{-1mm}$$
Similarly,   set   \mbox{$R_n:=  \{r \in R : \deg(r) = n\}$} and
\mbox{$\mathbf{r}_n := \vert R_n\vert$}; then
$\sum\limits_{r \in R} \mathbf{b}_{n-\deg(r)} =   \sum\limits_{i \in \integers}
\mathbf{r}_i {\cdot} \mathbf{b}_{n-i}$.
Now~\eqref{eq:ineqbis} becomes\vspace{-1mm}
\begin{equation}\label{eq:ineq2}
\textstyle
\mathbf{b}_n +  \sum\limits_{i \in \integers}   \mathbf{r}_i {\cdot}\mathbf{b}_{n-i} \ge
 \sum\limits_{i \in \integers}   \mathbf{x}_i {\cdot} \mathbf{b}_{n-i}.
\end{equation}

\section{Hilbert series}\label{sec:two}

Let $t$ be a new variable.  We shall express elements of the \vspace{-2mm} power-series ring \mbox{$\reals[[t]]$}
in the form \mbox{$\sum\limits_{n \in \integers}  a_nt^n $}, and understand that \mbox{$a_n = 0$} if \mbox{$n \le -1$}. Set
$$\textstyle P:= \{ \sum\limits_{n \in \integers} a_nt^n \in \reals[[t]]
: a_n \ge 0 \text{ for all } n \in \integers \}.\vspace{-1mm}$$
Then \mbox{$P$} is both an additive submonoid  and a multiplicative submonoid  in \mbox{$\reals[[t]]$}.
Let $\succeq$ be the relation on \mbox{$\reals[[t]]$} such that \mbox{$\alpha \succeq \beta$} if and only if
\mbox{$\alpha - \beta \in P$}.

\begin{remark}\label{rem:serre} In terms of this relation, the penultimate paragraph of the proof
of Theorem~\ref{thm:exact} says,
since $$
\textstyle(1-\lambda t){\cdot}(1 - \mu t){\cdot}(\sum\limits_{n \in \integers}
\mathbf{a}_{n}t^n) \,\,\,=\,\,\,(1 - \vert X \vert {\cdot} t + \vert R \vert{\cdot}t^2){\cdot}(\sum\limits_{n \in \integers}
\mathbf{a}_{n}t^n)\,\,\,\succeq \,\,\,   (1-t)^{-1},$$
$$\textstyle \sum\limits_{n \in \integers} \mathbf{a}_{n}t^n
\,\,\, \succeq \,\,\,   (1 - \mu t)^{-1}{\cdot}(1-\lambda t)^{-1} {\cdot}(1-t)^{-1} \,\,\, \succeq \,\,\,
(1 -  \mu t)^{-1} {\cdot}  (1-t)^{-1}.
$$
\end{remark}

Continuing with the notation developed in Section~\ref{sec:graded}, we henceforth restrict to the case where
 \mbox{$\mathbf{x}_n$}, \mbox{$\mathbf{r}_n \in \naturals$}; as in the proof of Theorem~\ref{thm:exact},
\mbox{$\mathbf{b}_n \in \naturals$}.
 We   define the \textit{Hilbert series} of $B$, $X$, and $R$,
 to be the elements of   \mbox{$\reals[[t]]$} given by
 \mbox{$\operatorname{H}(B):= \sum\limits_{n \in \integers} \mathbf{b}_nt^n$},
  \mbox{$\operatorname{h}(X):= \sum\limits_{n \in \integers} \mathbf{x}_nt^n$},
  and  \mbox{$\operatorname{h}(R):= \sum\limits_{n \in \integers} \mathbf{r}_nt^n$}, respectively.
Notice that the constant terms are $1$, $0$, and $0$, respectively.
Now~\eqref{eq:ineq2} says that \mbox{$\operatorname{H}(B) + \operatorname{h}(R) {\cdot} \operatorname{H}(B)
 \succeq \operatorname{h}(X) {\cdot} \operatorname{H}(B)  $}.  Hence,
\mbox{$\bigl(1 -  \operatorname{h}(X) + \operatorname{h}(R)\bigr){\cdot}\operatorname{H}(B) \succeq 0$}.
By considering the constant terms, we see that
\begin{equation}\label{eq:GS}
\bigl(1 -  \operatorname{h}(X) + \operatorname{h}(R)\bigr){\cdot}\operatorname{H}(B)\succeq 1;
\end{equation}
this is esentially  Lemma~2 of~\cite{GS}.
In fact, one can read directly from~\eqref{eq:ex} that
 $$\textstyle\bigl(1 -  \operatorname{h}(X) + \operatorname{h}(R)\bigr){\cdot}\operatorname{H}(B)
 - \operatorname{H}( \operatorname{Ker}  \partial) \,\, =\,\, \operatorname{H}(K)\,\,=\,\, 1.$$

\begin{key}\label{key:key}  Consider any \mbox{$\gamma  \in t{\cdot}\reals[[t]]$}.

If   \mbox{$\gamma  \succeq \operatorname{h}(R)$}, then
$
\bigl(1 -  \operatorname{h}(X)
+ \gamma \bigr){\cdot}\operatorname{H}(B)
 \,\,\, \succeq  \,\,\,
\bigl(1 -  \operatorname{h}(X) + \operatorname{h}(R)\bigr){\cdot}\operatorname{H}(B)
  \,\,\,   \succeq   \,\,\,    1.  $

If it is also the case that  \mbox{$ (1 -  \operatorname{h}(X) + \gamma \bigr)^{-1}  \succeq 0$},
then    \mbox{$\operatorname{H}(B)
 \,\,\,   \succeq   \,\,\,   \bigl(1 -  \operatorname{h}(X) + \gamma \bigr)^{-1}  \,\,\,  \succeq  \,\,\,  0.  $}

If it is further the case that $X$ is  finite and \mbox{$\gamma \ne  \operatorname{h}(X)$}, or, more generally,
that \mbox{$ (1 -  \operatorname{h}(X) + \gamma \bigr)^{-1} \not\in  \reals[t]$},
then   \mbox{$
\operatorname{H}(B)
$}   has infinitely many nonzero coefficients, and, hence,
  \mbox{$\dim_K(B)= \aleph_0$}.
\end{key}

\medskip

Finally, we restrict to the case where $X$ is concentrated in degree 1.

\begin{corollary}[{\normalfont{Golod}}\textbf]\label{cor:gol} Let $K$ be a field,  $X$ be a finite, nonempty set,
and \mbox{$\varepsilon$} be an element of
\mbox{$\left[0,\frac{\vert X \vert}{2}\, \right]$}. For each integer \mbox{$n \ge 2$},
let   $R_n$ be a  family of $X$-homogenous elements in
\mbox{$ K\langle X   \rangle $}
 of $X$-degree~$n$
such that \mbox{$\ \vert R_n \vert  \le  \varepsilon ^2 (\vert X \vert - 2\varepsilon )^{n-2}$}.
$\bigl($When \mbox{$\varepsilon  = \frac{\vert X \vert -1}{2}$}, this says
 \mbox{$ \vert R_n \vert  \le (\frac{ \vert X \vert -1  }{2})^2$}.$\bigr)$
Then \mbox{$\dim_K\bigl(K\langle X \mid \bigcup\limits_{n \ge 2} R_n \rangle\bigr)= \aleph_0$}.\vspace{-2mm}
\end{corollary}

\begin{proof} Set \mbox{$\gamma    := \textstyle
\sum\limits_{n \ge 2} \bigl(\varepsilon ^2 (\vert X \vert - 2\varepsilon )^{n-2} t^n\bigr) =
\sum\limits_{m \ge 0} \bigl(\varepsilon ^2 (\vert X \vert - 2\varepsilon )^{m} t^mt^2\bigr) =
 \frac{\varepsilon ^2 t^2}{1- (\vert X \vert - 2\varepsilon )t}$}. \vspace{.5mm}

Set
\mbox{$\alpha:= 1- (\vert X \vert - \varepsilon) t$}  \, and\,   \mbox{$\beta:=  \varepsilon t$}.
Then  \mbox{$\alpha-\beta = 1 -  \vert X \vert t$},\,\,  \mbox{$\alpha +\beta =
1- (\vert X \vert - 2\varepsilon )t$}
,\,   \mbox{$\gamma
=  \frac{\beta^2}{\alpha +\beta}$},
\begin{align*}
& \textstyle 1 -  \vert X \vert t + \gamma  =
 (\alpha-\beta) +  \frac{\beta^2}{\alpha +\beta}
= \frac{\alpha^2}{ \alpha+\beta},
 \quad\text{and}\quad \bigl(1 -  \vert X \vert t + \gamma )^{-1}  =   \frac{1}{\alpha} + \frac{\beta}{\alpha^{2}}.
\end{align*}
The result now follows by~\ref{key:key}, since \mbox{$\vert X \vert > \varepsilon \ge 0$}.
\end{proof}

\begin{history} Golod~\cite{G} then used Corollary~\ref{cor:gol} to create new phenomena:
  finitely generated, non-nilpotent, nil algebras and
 infinite, residually finite, finitely generated    $p$-groups, for each prime $p$.
These were the first  infinite, finitely generated torsion groups.\end{history}

\section{The Fox resolution for group algebras}

We now recall the group-algebra analogue of the Koszul resolution.

 \begin{notation} Let
$G$ be a group.
Let \mbox{$\operatorname{d}(G)$} denote the smallest of those cardinals $\kappa$ such that $G$ can be generated by $\kappa$ elements.
Let \mbox{$G'$} denote the derived subgroup of $G$, and set \mbox{$G^{\text{ab}}:= G/G'$}, the
abelianization of $G$.

Let   \mbox{$\langle X \mid R \rangle$} be a presentation for $G$.
 Clearly,  \mbox{$\operatorname{d}(G^{\text{ab}}) \le \operatorname{d}(G) \le \vert X \vert$}.

Let $K$ be a field and set  \mbox{$B:=KG$}, the group algebra.  Let \mbox{$\mathfrak{b}$} denote the
kernel of  the $K$-algebra homomorphism
\mbox{$B \to K$} which carries $G$ to $\{1\}$.
Let $F$ be the  group algebra over $K$
for the free group on $X$,
   $\mathfrak{f}$~be the two-sided
ideal of $F$ generated by \mbox{$\{x-1 \mid x \in X\}$}, and
$\mathfrak{r}$ be the two-sided ideal of $F$
generated by \mbox{$\{r-1 \mid r \in R\}$}.  Then \mbox{$B  = F/\mathfrak{r}$}  and
\mbox{$\mathfrak{b} = \mathfrak{f}/\mathfrak{r}$}.

Set \mbox{$K^{(X)} := \bigoplus\limits_{x \in X}Kx$},
\mbox{$F^{(X)}:= F \otimes_K K^{(X)}$}, and \mbox{$B^{(X)}:= B \otimes_K K^{(X)}$}, and similarly
with $R$ in place of $X$.
\end{notation}

\begin{definitions}\label{defs:groupring} It is not difficult to see that the left ideal of $F$ generated by
\mbox{$\{x-1 \mid x \in X\}$} is closed under right multiplication by the elements of
\mbox{$X \cup X^{-1}$}, and, hence, is the whole of   \mbox{$\mathfrak{f}$}.  We have a left-$F$-module map
  \mbox{$F^{(X)} \to \mathfrak{f}$} which sends each \mbox{$1 \otimes x$} to \mbox{$x-1$};
to construct an inverse, we shall define a left-$F$-module map
  \mbox{$ \mathfrak{f} \to F^{(X)}$} which sends each \mbox{$x-1$} to \mbox{$1 \otimes x$}.

We view \mbox{$F^{(X)}$} as an \mbox{$(F,K)$}-bimodule, and form the
bimodule-algebra over $K$  suggestively written in matrix form as
\mbox{$\left(\begin{smallmatrix}
F &F^{(X)}\\ 0 &K
\end{smallmatrix}\right)$}.  There then exists a unique $K$-algebra homomorphism
\mbox{$\left(\begin{smallmatrix}
\phi_{1,1} &\phi_{1,2}\\ 0 &\phi_{2,2}
\end{smallmatrix}\right): F \to \left(\begin{smallmatrix}
F &F^{(X)}\\ 0 &K
\end{smallmatrix}\right)$} which sends each \mbox{$x \in X$} to
the invertible element
\mbox{$\left(\begin{smallmatrix}
x &1 \otimes x\\ 0 &1
\end{smallmatrix}\right)$}.  Thus, the \mbox{$\phi_{i,j}$} are $K$-module maps, \mbox{$\phi_{1,1}(1) = 1$},  \mbox{$\phi_{1,2}(1) = 0$},
\mbox{$\phi_{2,2}(1) = 1$}, and, for all $f$, \mbox{$g \in F$},
$$\phi_{1,1}(f {\cdot}g) = \phi_{1,1}(f) {\cdot}\phi_{1,1}(g),\hskip 5pt
 \phi_{1,2}(f {\cdot}g) = \phi_{1,1}(f)  {\cdot} \phi_{1,2}(g)+\phi_{1,2}(f)  {\cdot} \phi_{2,2}(g),
  \text{  and }
 \phi_{2,2}(f {\cdot}g) = \phi_{2,2}(f) {\cdot} \phi_{2,2}(g).$$  In particular, \mbox{$\phi_{1,1}$} and \mbox{$\phi_{2,2}$} are
$K$-algebra homomorphisms.  Also, for all \mbox{$x \in X$},
$$\phi_{1,1}(x) = x, \hskip5pt\phi_{1,2}(x) = 1\otimes x, \text{ and } \phi_{2,2}(x) = 1.$$
In particular, \mbox{$\phi_{1,1}$} is the identity map on $F$, and \mbox{$\phi_{2,2}(\mathfrak{f}) = \{0\}$}.
Hence, \mbox{$\phi_{1,2}$} restricted to \mbox{$\mathfrak{f}$}  is
a left $F$-module map \mbox{$\mathfrak{f} \to F^{(X)}$} which sends each
\mbox{$x-1$} to   \mbox{$ 1\otimes x$}, as desired.  Now
each element $f$ of $\mathfrak{f}$ has a unique
expression  as a left
$F$-linear combination of the elements of \mbox{$\{x-1 \mid x \in X\}$}, which we write as
\mbox{$f= \sum\limits_{x \in X} \frac{\partial f}{\partial (x-1)}{\cdot} (x-1)$}.\vspace{-2mm}

We have an isomorphism of left $F$-modules
\begin{equation*}
\textstyle  \mathfrak{f}  \xrightarrow{\sim}  F^{(X)}, \qquad
 f = \sum\limits_{x\in X} \frac{\partial f}{\partial (x-1)}{\cdot}
 (x-1) \mapsto \textstyle \sum\limits_{x\in X} \frac{\partial f}{\partial (x-1)} \otimes x.\vspace{-1mm}
\end{equation*}
On applying    \mbox{$(F/\mathfrak{r})\otimes_F -$}, we obtain an isomorphism of left
 \mbox{$ F/\mathfrak{r}$}-modules  $$\textstyle \mathfrak{f}/\mathfrak{rf}  \xrightarrow{\sim}   B^{(X)},
\qquad f + \mathfrak{rf} \mapsto \textstyle \sum\limits_{x\in X}
 (\mkern-2mu\frac{\partial f}{\partial (x-1)}+\mathfrak{r} )\otimes x.$$
\vspace{-4mm}

We  have also a surjection of $F$-bimodules
\begin{equation*}
\textstyle F \otimes_K K^{(R)} \otimes_ K F  \onto \mathfrak{r}. \qquad
f_1 \otimes r \otimes f_2 \mapsto f_1{\cdot}(r-1){\cdot}f_2.
\end{equation*}
On applying    \mbox{$(F/\mathfrak{r}) \otimes_F - \otimes_F (F/\mathfrak{f})$},  we obtain
a surjection of left \mbox{$F/\mathfrak{r}$}-modules
 $$\textstyle B^{(R)}
  \onto    \mathfrak{r}/\mathfrak{rf}, \qquad   (f+\mathfrak{r})\otimes r \mapsto f{\cdot}(r-1) + \mathfrak{rf}.$$

 The cokernel of the composite \mbox{$ B^{(R)}
  \onto    \mathfrak{r}/\mathfrak{rf} \into  \mathfrak{f}/\mathfrak{rf} \xrightarrow{\sim}   B^{(X)}$}
is isomorphic to \mbox{$\mathfrak{f}/\mathfrak{r}$}, which is \mbox{$\mathfrak{b}$}.  We then have
an exact left-$B$-module sequence\vspace{-1mm}
\begin{equation}\label{eq:group}
\textstyle  B^{(R)}  \xrightarrow{ \partial:b\otimes r \mapsto \sum_{x \in X}
 b{\cdot}(\mkern-2mu\frac{\partial (r-1)}{\partial (x-1)}+\mathfrak{r})\otimes   x}
 B^{(X)}
\xrightarrow{  b \otimes x \mapsto b{\cdot}(x-1+\mathfrak{r})} \mathfrak{b} \to 0.
\end{equation}
\end{definitions}

\begin{theorem}\label{thm:new}  Let \mbox{$\langle X \mid R \rangle$} be
a presentation of a nontrivial, finite group $G$.
If  \mbox{$\vert X \vert = \operatorname{d}(G^{\text{\normalfont ab}})$},
then    \mbox{$\vert R \vert > \frac{1}{4}\vert X \vert^2$}; equivalently,
if  \mbox{$\operatorname{d}(G) = \operatorname{d}(G^{\text{\normalfont ab}}) =
\vert X \vert$}, then
   \mbox{$\vert R \vert > \frac{1}{4} (\operatorname{d}(G))^2$}.
\end{theorem}

\vspace{-5mm}

\begin{proof} Since \mbox{$ G^{\text{ab}}$}  is a finite abelian group,  \mbox{$G^{\text{ab}} \simeq \bigoplus\limits_{i=1}^d (\integers/I_i)$}
for some    finite  chain \mbox{$I_1 \subseteq I_2 \subseteq \cdots \subseteq I_d$} of proper  ideals of $\integers$.
Let  $p$ be a prime number such that \mbox{$I_d \subseteq p\integers$}, and set   \mbox{$K:= \integers/p\integers$}.
Then  \mbox{$ K\otimes_{\integers} G^{\text{ab}} \simeq K^d$}, and
\mbox{$d = \dim_K(K\otimes_{\integers} G^{\text{ab}})   \le \operatorname{d}(G^{\text{ab}}) \le d$}.  Thus,
\mbox{$\dim_K(K\otimes_{\integers} G^{\text{ab}}) = \operatorname{d}(G^{\text{ab}}) =  \vert X \vert$}.

   It is well known and straightforward
to prove that \mbox{$\mathfrak{b}/\mathfrak{b}^2 \simeq  K\otimes_{\integers} G^{\text{ab}} $} with
  \mbox{$\textstyle (g{-}1) + \mathfrak{b}^2  \leftrightarrow 1\otimes gG'$}.
Then   \mbox{$\dim_K\mathfrak{b}/\mathfrak{b}^2 = \dim_K(K\otimes_{\integers} G^{\text{ab}}) =   \vert X \vert$}.
The result now follows from Theorem~\ref{thm:exact} applied to~\eqref{eq:group}.
\end{proof}

\begin{corollary}[{\normalfont  Golod-Shafarevich}]    Let $p$ be a prime number,
and  \mbox{$\langle X \mid R \rangle$} be
a presentation of a nontrivial,  finite $p$-group $G$.
If
\mbox{$\vert X \vert=\operatorname{d}(G)$}, then
   \mbox{$\vert R \vert > \frac{1}{4} (\operatorname{d}(G))^2$}.
\end{corollary}

 \begin{proof}  By the Burnside  basis theorem,  \mbox{$\operatorname{d}(G)
= \dim_K(K\otimes_{\integers} G^{\text{ab}})$} for \mbox{$K = \integers/p\integers$}.
Hence,  \mbox{$\operatorname{d}(G) = \operatorname{d}(G^{\text{ab}})$}, and the
 result    follows from the second part of Theorem~\ref{thm:new}.
 \end{proof}

\begin{definitions}[continued]
For \mbox{$f \in F{-}\{0\}$}, we set
\mbox{$\deg(f):= \max \{ i \in \naturals : f \in \mathfrak{f}^{\mkern1mu i} \}$}.
For each $r \in R$, we have then defined  \mbox{$\deg(r{-}1) \in \naturals-\{0\}$}, unless $r = 1$ in~$F$,
in which case we shall choose some  value  \mbox{$\deg(r{-}1) \in \naturals-\{0\}$}.
Then \mbox{ $\frac{\partial (r-1)}{\partial (x-1)}  \in \mathfrak{f}^{\,\deg(r-1) - 1}$}, for each \mbox{$x \in X$}.

Let $n$ range over $\integers$.
Define \mbox{$\mathfrak{b}^{n}$} to be $B$ if \mbox{$n\le 0$}, and,
as usual,  to be  \mbox{$\mathfrak{b}^{n-1}{\cdot}\mathfrak{b}$} if \mbox{$n \ge 1$}.
In~\eqref{eq:group}, we find, for each \mbox{$r \in R$},     $$\textstyle \partial (\mathfrak{b}^n \otimes_K Kr) \subseteq
 \bigoplus\limits_{x \in X}  (\mathfrak{b}^{n+\text{deg}(r-1)-1}  \otimes_K Kx), \text{ and, hence, }
 \partial (\mathfrak{b}^{n-\text{deg}(r-1)+1} \otimes_K Kr) \subseteq
 \bigoplus\limits_{x \in X}  (\mathfrak{b}^{n}  \otimes_K Kx).$$
On applying \mbox{$(B/\mathfrak{b}^{n})\otimes_B-$} to~\eqref{eq:group}, we get
an exact left-$B$-module sequence
\begin{equation*}
\textstyle  \bigoplus\limits_{r \in R} (B/\mathfrak{b}^{n-\text{deg}(r-1)+1}) \otimes_K Kr
\to
 \bigoplus\limits_{x \in X} (B/\mathfrak{b}^{n}) \otimes_K Kx
 \to (\mathfrak{b}+\mathfrak{b}^{n+1})/\mathfrak{b}^{n+1}  .
\end{equation*}
Set \mbox{$\mathbf{a}_n:= \dim_K(B/\mathfrak{b}^{n+1})$}, and
define \mbox{$\delta_n$} to be $0$ if \mbox{$n \le -1$} and to be~$1$ if \mbox{$n  \ge 0$}. Since $K$ is a field, \mbox{$
\textstyle \vert X \vert \mathbf{a}_{n-1}  \le (\sum\limits_{r \in R} \mathbf{a}_{n-\text{deg}(r-1)})
+ (\mathbf{a}_{n} - \delta_{n}).$}

 We set \mbox{$R_n:= \{ r \in R : \deg(r-1) = n\}$} and \mbox{$\mathbf{r}_n := \vert R_n\vert$}.
We henceforth restrict to the case where
 \mbox{$\vert X \vert$},~\mbox{$\mathbf{r}_n \in \naturals$}; as in the proof of Theorem~\ref{thm:exact},
\mbox{$\mathbf{a}_n \in \naturals$}.
We define   \mbox{$\operatorname{h}(R):= \sum\limits_{n \in \integers} \mathbf{r}_nt^n\in \reals[[t]]$}. We
define  \mbox{$\operatorname{h}(X)$} similarly, and find \mbox{$\operatorname{h}(X) = \vert X \vert t$}.
Set\vspace{-2mm}
\mbox{$\mathbf{b}_n:= \dim_K(\mathfrak{b}^{n}/\mathfrak{b}^{n+1}) = \mathbf{a}_n - \mathbf{a}_{n-1}$}
and   \mbox{$\operatorname{H}(B):= \sum\limits_{n \in \integers} \mathbf{b}_nt^n$}.
Notice that   \mbox{$(1-t){\cdot}\sum\limits_{n \in \integers} \mathbf{a}_nt^n = \operatorname{H}(B)$}
and \mbox{$\sum\limits_{n \in \integers} \delta_nt^n = (1-t)^{-1}$}.  Now \vspace{-1.5mm}
$$\bigl(1 - \operatorname{h}(X) + \operatorname{h}(R)\bigr) {\cdot}
(1-t)^{-1}{\cdot} \operatorname{H}(B) \succeq (1-t)^{-1}.$$
This is the form of Vinberg's inequality for filtered algebras~\cite{V}. The method of proof outlined here
 is based on the proof of Theorem~\ref{thm:exact} above, which, in turn,  may have been suggested by Vinberg's work.

Set \mbox{$\alpha:= 1 - \operatorname{h}(X) + \operatorname{h}(R)
\in \reals[[t]]$}. We claim that if there exists some
\mbox{$\varepsilon \in [0,1]$} such that the real series resulting
from replacing the $t$s in \mbox{$\alpha$} with $\varepsilon$s  converges
to a value \mbox{$\alpha(\varepsilon)\in\,  \left]-\infty,0\right]$}, then
 \mbox{$\operatorname{H}(B)   \not \in   \reals[t]$}, and, in particular,
  \mbox{$\dim_K(B) = \aleph_0$}.  This is clear if \mbox{$\varepsilon \ne 1$}.
If  \mbox{$\varepsilon = 1$},  then \mbox{$\alpha  \in \integers[t]$};
here, if \mbox{$\alpha(1) \ne 0$}, we may replace \mbox{$\varepsilon$} with a value slightly smaller than $1$
to pass to the preceding case,
while if \mbox{$\alpha(1) = 0$}, then \mbox{$\alpha {\cdot}(1-t)^{-1} \in \integers[t]$}, and the
desired conclusion holds.  Thus the claim is proved.
If  \mbox{$\vert X \vert \ge 1$}, \mbox{$\mathbf{r}_1 = 0$}, and \mbox{$\vert R \vert \le \frac{1}{4}\vert X \vert^2$},
then  we may take \mbox{$\varepsilon$} to be \mbox{$\min\{\frac{2}{\vert X \vert}, 1\}$} to
 recover Theorem~\ref{thm:new}.  Notice that
 \mbox{$\mathbf{r}_1 = 0$} if and only if \mbox{$\mathbf{b}_1 = \vert X \vert$}, and
 if \mbox{$ \vert X \vert = 1$}, then \mbox{$\lfloor \frac{1}{4}\vert X \vert^2\rfloor  =  0 $}.
 \end{definitions}

\begin{history}  Suppose that \mbox{$G = \langle X \mid R \rangle$} is a group presentation such
 that  \mbox{$\operatorname{d}(G^{\text{\normalfont ab}}) = \vert X \vert < \aleph_0$}.
 Theorem~\ref{thm:new} says that if \mbox{$\vert R \vert \le \frac{1}{4} \vert X \vert^2$},
then  either $G$ is trivial $\bigl($where \mbox{$(\vert X \vert, \vert R \vert) = (0,0)\bigr)$} or $G$ is infinite.
  Wilson~\cite{Wilson} showed   that if \mbox{$\vert R \vert < \frac{1}{4} \vert X \vert^2$},
then either $G$ is infinite cyclic $\bigl($where \mbox{$(\vert X \vert, \vert R \vert) = (1,0)\bigr)$}
or $G$ maps onto a residually finite, infinite   $p$-group,
for some prime $p$.
His proof is based on Vinberg's inequality and the methods of Golod~\cite{G}.
 A  recent introduction to related results can be found in~\cite{Ershov}.
\end{history}

\end{document}